\documentclass[a4paper,12pt]{article}
\usepackage{amsmath}
\usepackage{amsthm}
\usepackage{amssymb}
\usepackage{enumerate}
\usepackage{url}

\title{A unification of the hypercontractivity and its exponential variant of the Ornstein-Uhlenbeck semigroup}

\author{Yuu Hariya\thanks{Mathematical Institute, 
Tohoku University, Aoba-ku, Sendai 980-8578, Japan. }}

\date{\empty}
\numberwithin{equation}{section}

\theoremstyle{plain}

\newtheorem{thm}{Theorem}[section]

\newtheorem{prop}{Proposition}[section]

\newtheorem{cor}{Corollary}[section]

\newtheorem{lem}{Lemma}[section]

\theoremstyle{definition}

\theoremstyle{remark}
\newtheorem{rem}{Remark}[section]
\newtheorem{exm}{Example}[section]

\DeclareMathOperator*{\esssup}{ess\,sup}
\DeclareMathOperator*{\essinf}{ess\,inf}

\begin{document}

\def\N {\mathbb{N}}
\def\R {\mathbb{R}}
\def\Q {\mathbb{Q}}

\def\F {\mathcal{F}}

\def\kp {\kappa}

\def\ind {\boldsymbol{1}}

\def\al {\alpha }
\def\la {\lambda }
\def\ve {\varepsilon}
\def\Om {\Omega}

\def\ga {\gamma }

\newcommand\ND{\newcommand}
\newcommand\RD{\renewcommand}

\ND\lref[1]{Lemma~\ref{#1}}
\ND\tref[1]{Theorem~\ref{#1}}
\ND\pref[1]{Proposition~\ref{#1}}
\ND\sref[1]{Section~\ref{#1}}
\ND\ssref[1]{Subsection~\ref{#1}}
\ND\aref[1]{Appendix~\ref{#1}}
\ND\rref[1]{Remark~\ref{#1}} 
\ND\cref[1]{Corollary~\ref{#1}}
\ND\eref[1]{Example~\ref{#1}}
\ND\fref[1]{Fig.\ {#1} }
\ND\lsref[1]{Lemmas~\ref{#1}}
\ND\tsref[1]{Theorems~\ref{#1}}
\ND\dref[1]{Definition~\ref{#1}}
\ND\psref[1]{Propositions~\ref{#1}}
\ND\rsref[1]{Remarks~\ref{#1}}
\ND\sssref[1]{Subsections~\ref{#1}}

\ND\pr{\mathbb{P}}
\ND\ex{\mathbb{E}}
\ND\br{W}

\ND\gss[1]{\gamma_{#1}}
\ND\ou{Q}
\ND\norm[2]{\left\|{#1}\right\|_{#2}}
\ND\lp[2]{L^{#1}({#2})}
\ND\calB{\mathcal{B}}
\ND\vp{\varphi}
\ND\U{U}
\ND\V{V}
\ND\ro{\rho (1)}
\ND\dv{\theta}
\ND\cf{c}
\ND\uf{u}
\ND\vf{v}
\ND\su{\mathsf{u}}
\ND\sv{\mathsf{v}}
\ND\G{G}
\ND\h{H}
\ND\tu{\tilde{\su}}
\ND\tv{\tilde{\sv}}
\ND\C[1]{C^{1}_{b,0}(#1)}
\ND\D{d}

\def\thefootnote{{}}

\maketitle 
\begin{abstract}
Let $\gss{\D }$ be the $\D $-dimensional standard Gaussian 
measure and $\{ \ou _{t}\} _{t\ge 0}$ the Ornstein-Uhlenbeck 
semigroup acting on $\lp{1}{\gss{\D }}$. 
We show that the hypercontractivity of 
$\{ \ou _{t}\} _{t\ge 0}$ is equivalent to the property that  
\begin{align*}
 \left\{ 
 \int _{\R ^{\D }}\exp \left( e^{2t}\ou _{t}f\right) d\gss{\D }
 \right\} ^{1/e^{2t}}
 \le \int _{\R ^{\D }}e^{f}\,d\gss{\D }, 
\end{align*}
which holds for any $f\in \lp{1}{\gss{\D }}$ with 
$e^{f}\in \lp{1}{\gss{\D }}$ and for any $t\ge 0$. 
We then derive a family of inequalities that unifies this 
exponential variant and the original hypercontractivity; 
a generalization of the Gaussian logarithmic 
Sobolev inequality is obtained as a corollary. 
A unification of the reverse hypercontractivity and the exponential variant 
is also provided. 
\footnote{E-mail: hariya@m.tohoku.ac.jp}
%%\footnote{{\itshape Running head}:~}
\end{abstract}
%% 47D07 Markov semigroups and applications to diffusion processes 
%% 60H30 Applications of stochastic analysis 
%% 60J65 Brownian motion 

%%%%%% New section %%%%%%
\section{Introduction}\label{;intro}
For a given positive integer $\D $, we denote by $\gss{\D }$ the standard Gaussian 
measure on $(\R ^{\D },\calB (\R ^{\D }))$ with $\calB(\R ^{\D })$ the Borel 
$\sigma $-field on $\R ^{\D }$. For every $p>0$, define 
\begin{align*}
 \lp{p}{\gss{\D }}:=
 \left\{ 
 f:\R ^{\D }\to \R ;\,\text{$f$ is measurable and satisfies }
 \int _{\R ^{\D }}|f(x)|^{p}\,\gss{\D }(dx)<\infty 
 \right\} 
\end{align*} 
and set 
\begin{align*}
 \norm{f}{p}:=
 \left\{ 
 \int _{\R ^{\D }}|f(x)|^{p}\,\gss{\D }(dx)
 \right\} ^{1/p}, \quad f\in \lp{p}{\gss{\D }}. 
\end{align*}
(Although $\norm{\cdot }{p}$ is not a norm for $p<1$, we abuse the common 
notation in the sequel.) We denote by $\ou =\{ \ou _{t}\} _{t\ge 0}$ the 
Ornstein-Uhlenbeck semigroup acting on $\lp{1}{\gss{\D }}$: for 
$f\in \lp{1}{\gss{\D }}$ and $t\ge 0$, 
\begin{align*}
 \left( \ou _{t}f\right) \!(x)
 :=\int _{\R ^{\D }}
 f\left( e^{-t}x+\sqrt{1-e^{-2t}}y\right) \gss{\D }(dy), \quad x\in \R ^{\D };  
\end{align*}
note that $|\ou _{t}f|<\infty $ $\gss{\D }$-a.e.\ when 
$f\in \lp{1}{\gss{\D }}$, since there holds the identity 
\begin{align*}
 \int _{\R ^{\D }}\ou _{t}|f|\,d\gss{\D }
 =\int _{\R ^{\D }}|f|\,d\gss{\D }
\end{align*}
for any measurable function $f$ on $\R ^{\D }$. 
It is well known that $\ou $ enjoys the hypercontractivity: if 
$f\in \lp{p}{\gss{\D }}$ for some $p>1$, then 
\begin{align}\label{;hc}
 \norm{\ou _{t}f}{q(t)}\le \norm{f}{p} \quad \text{for all }t\ge 0, 
\end{align}
with $q(t)=e^{2t}(p-1)+1$. 

The hypercontractivity \eqref{;hc} was firstly 
observed by Nelson \cite{nel} and applied in quantum field theory; it was 
found later by Gross \cite{gro} to be equivalent to the (Gaussian) logarithmic 
Sobolev inequality: 
\begin{align}\label{;lsi}
 \int _{\R ^{\D }}|f|^{2}\log |f|\,d\gss{\D }
 \le \int _{\R ^{\D }}
 |\nabla f|^{2}\,d\gss{\D }
 +\norm{f}{2}^{2}\log \norm{f}{2}, 
\end{align} 
where $f$ is any weakly differentiable function 
in $\lp{2}{\gss{\D }}$ with $|\nabla f|\in \lp{2}{\gss{\D }}$. Because 
of their dimension-free formulation, the hypercontractivity \eqref{;hc}
as well as the logarithmic Sobolev inequality \eqref{;lsi} have importance in the 
Malliavin calculus; see, e.g., the monograph \cite{shi} by Shigekawa. We also 
remark that the Gaussian logarithmic Sobolev inequality \eqref{;lsi} goes back to 
Stam \cite{sta}, on which we refer the reader to \cite[Section~8.13]{ll}. 

In this paper, we show the equivalence between the hypercontractivity 
\eqref{;hc} and the following property of $\ou $: for any 
$f\in \lp{1}{\gss{\D }}$ with $e^{f}\in \lp{1}{\gss{\D }}$, it holds that 
\begin{align}\label{;ehc}
 \norm{\exp \left( \ou _{t}f\right) }{e^{2t}}
 \le \| e^{f}\| _{1} \quad \text{for all }t\ge 0. 
\end{align}

\begin{prop}\label{;pequiv}
 The hypercontractivity \eqref{;hc} and the property \eqref{;ehc} are 
 equivalent. 
\end{prop}

In fact, by using Jensen's inequality, it is easily seen that \eqref{;hc} 
implies \eqref{;ehc}; on the other hand, it can also be seen that 
\eqref{;ehc} implies the logarithmic Sobolev inequality \eqref{;lsi}, 
and hence implies \eqref{;hc} thanks to the above-mentioned 
equivalence between \eqref{;hc} and \eqref{;lsi}.  

We then show that the above two properties \eqref{;hc} and \eqref{;ehc} of 
$\ou $ are unified into 

\begin{thm}\label{;tunif}
 Let a positive function $\cf :(0,\infty )\to (0,\infty )$ be in 
 $C^{1}((0,\infty ))$ and satisfy 
 \begin{align}\label{;cond1}\tag{C}
  \text{$\cf '>0$ and 
  $\dfrac{\cf}{\cf '}$ is concave on $(0,\infty )$.}
 \end{align}
 We set 
 \begin{align}
  \uf (t,x)=\int _{0}^{x}\cf (y)^{e^{2t}}\,dy, 
  \quad t\ge 0,\ x>0. \label{;defuf}
 \end{align}
 Then for any nonnegative, measurable function $f$ on $\R ^{\D }$ such that 
 \begin{align}\label{;intf}
  \uf (0,f)\in \lp{1}{\gss{\D }}, 
 \end{align}
 we have 
 \begin{align}\label{;genhc}
  \vf \left( t,\norm{\uf (t,\ou _{t}f)}{1}\right) 
  \le \vf \left( 0,\norm{\uf (0,f)}{1}\right) \quad \text{for all }t\ge 0. 
 \end{align}
 Here for every $t\ge 0$, the function 
 $\vf (t,\cdot )$ is the inverse function of 
 $\uf (t,x),\,x>0$. 
\end{thm}

It is easily checked that two functions $x^{p-1}$ with $p>1$ and 
$e^{x}$ fulfill the condition \eqref{;cond1}; in fact, they both 
satisfy $(\cf /\cf ')''=0$. 
These two choices of 
$\cf $ in \tref{;tunif} lead to \eqref{;hc} and \eqref{;ehc}, 
respectively; see \rref{;runif} for details. For other examples of $\cf $ 
satisfying \eqref{;cond1}, see \eref{;cexample}. We also show that 
differentiating the left-hand side of \eqref{;genhc} at $t=0$ gives us 
a generalization of the logarithmic Sobolev inequality \eqref{;lsi}; 
see \cref{;cglsi}. 

\begin{rem}\label{;rtunif}
\thetag{1} The condition \eqref{;intf} imposed on a nonnegative $f$ 
implies $f\in \lp{1}{\gss{\D }}$, hence $\ou _{t}f$ in 
\eqref{;genhc} is well-defined. To see this, note that 
the positivity and concavity of $\cf /\cf '$ entail that 
there exist positive constants $\kp _{1},\kp _{2}$ such that 
\begin{align*}
 \frac{\cf }{\cf '}(x)\le \kp _{1}x+\kp _{2} \quad \text{for all }x>0, 
\end{align*}
from which it follows that for all 
$x,y>0$ with $x>y$, 
\begin{align*}
 \frac{\cf (x)}{\cf (y)}
 \ge 
 \left( 
 \frac{\kp _{1}x+\kp _{2}}{\kp _{1}y+\kp _{2}}
 \right) ^{1/\kp _{1}}. 
\end{align*}
Therefore if \eqref{;intf} is fulfilled, then we have 
$f\in \lp{1+1/\kp _{1}}{\gss{\D }}$ for some $\kp _{1}>0$. 

\noindent 
\thetag{2} The left-hand side of \eqref{;genhc} is nonincreasing 
in $t$: 
\begin{align}\label{;incr}
 \vf \left( t+s,\norm{\uf (t+s,\ou _{t+s}f)}{1}\right) 
  \le \vf \left( s,\norm{\uf (s,\ou _{s}f)}{1}\right) 
\end{align}
for any $s,t\ge 0$. To see \eqref{;incr}, fix $s\ge 0$ and set for 
$t\ge 0$ and $x>0$, 
\begin{align*}
 \cf ^{s}(x)=\left\{ \cf (x)\right\} ^{e^{2s}}
 && \text{and} && 
 \uf ^{s}(t,x)=\int _{0}^{x}\left\{ \cf ^{s}(y)\right\} ^{e^{2t}}dy
 \equiv \uf (t+s,x). 
\end{align*}
We also write $\vf ^{s}(t,\cdot )$ for the inverse function of 
$\uf ^{s}(t,\cdot )$. Since $\cf ^{s}/(\cf ^{s})'=e^{-2s}\cf /\cf '$, 
the function $\cf ^{s}/(\cf ^{s})'$ is also concave on 
$(0,\infty )$. Therefore we may apply \tref{;tunif} to 
$\ou _{s}f$ with replacing $\cf $, $\uf $ and $\vf $ therein by 
$\cf ^{s}$, $\uf ^{s}$ and $\vf ^{s}$, respectively, to get 
\begin{align*}
 \vf ^{s}\left( t,\norm{\uf ^{s}(t,\ou _{t}(\ou _{s}f))}{1}\right) 
  \le \vf ^{s}\left( 0,\norm{\uf ^{s}(0,\ou _{s}f)}{1}\right) , 
\end{align*}
which is \eqref{;incr} thanks to the identities 
$\uf ^{s}(t,\cdot )=\uf (t+s,\cdot )$, 
$\vf ^{s}(t,\cdot )=\vf (t+s,\cdot )$ and the semigroup property 
$\ou _{t}(\ou _{s}f)=\ou _{t+s}f$. The monotonicity 
\eqref{;incr} may also be seen directly from the proof of the theorem 
given in \sref{;sprft}. 
\end{rem}

It will also be shown that if we replace \eqref{;cond1} by 
the condition that $\cf '<0$ and $\cf /\cf '$ is convex 
on $(0,\infty )$, then the concluding inequality 
\eqref{;genhc} is reversed, yielding in particular the 
reverse hypercontractivity of $\ou $: 
if a $\gss{\D }$-a.e.\ positive $f\in \lp{1}{\gss{\D }}$ satisfies 
$1/f\in \lp{\al }{\gss{\D }}$ for some $\al >0$, then it holds that 
\begin{align}\label{;rhc}
 \norm{1/\ou _{t}f}{e^{2t}(\al +1)-1}\le \norm{1/f}{\al } 
 \quad \text{for all }t\ge 0. 
\end{align}
See \sref{;srhc}. 

We remark that since there 
are not involved any constants dependent on the dimension $\D $, every 
result mentioned above can be extended to the framework of abstract 
Wiener space through finite-dimensionalization. 
\smallskip 

We give an outline of the paper. In \sref{;spre} we provide preliminary 
lemmas. In \sref{;sprfp} we prove \pref{;pequiv}. 
In \sref{;sprft}, we give a proof of \tref{;tunif} as well as examples 
of functions $\cf $ satisfying the assumption of the theorem. 
As a corollary to \tref{;tunif}, we also derive a family of 
inequalities that includes the logarithmic Sobolev inequality \eqref{;lsi} 
as a particular case. In the final section, we show a unification 
of the reverse hypercontractivity \eqref{;rhc} and the 
exponential variant \eqref{;ehc} of the hypercontractivity; 
some related inequalities are also presented. 
\medskip 

In the sequel, we denote by $x\cdot y$ the inner product 
of $x$ and $y$ in $\R ^{\D }$ 
and by $|x|$ the Euclidean norm of $x$: $|x|=\sqrt{x\cdot x}$. 
Given a positive integer $m$, the symbol 
$C^{1}_{b}(\R ^{m})$ stands for the set of bounded 
$C^{1}$-functions on $\R ^{m}$ with bounded derivatives.
 We denote by $\C{\R ^{m}}$ 
the set of functions $f$ in $C^{1}_{b}(\R ^{m})$ bounded away 
from $0$: $\inf \limits_{x\in \R ^{m}}f(x)>0$. For a given multivariate 
function $g(t,x)$, its subscripts denote partial differentiations: 
$g_{x}(t,x)=(\partial g/\partial x)(t,x)$, 
$g_{tx}(t,x)=(\partial ^{2}g/\partial x\partial t)(t,x)$, and so on. 
For two functions $h_{1}(z),h_{2}(z)$ in a variable $z$, we often write 
$(h_{1}/h_{2})(z)$ to denote $h_{1}(z)/h_{2}(z)$. Other notation will be 
introduced as needed. 

%%%%%% New Section %%%%%%
\section{Preliminaries}\label{;spre}
In this section, we state and prove preliminary lemmas. 
For this purpose, we prepare a probability space 
$(\Om, \F ,\pr )$ on which a $\D $-dimensional standard 
Brownian motion $\br =\{ \br _{t}\} _{0\le t\le 1}$ is defined. 
We denote by $\{ \F _{t}\} _{0\le t\le 1}$ the (augmentation of) 
the natural filtration of $\br $. Pick $f\in \lp{1}{\gss{\D }}$ 
and set 
\begin{align*}
 M_{t}\equiv M_{t}(f):=\ex \left[ 
 f(W_{1})|\F _{t}
 \right] ,\quad 0\le t\le 1. 
\end{align*}
Then by the Markov property of $\br $, 
\begin{align*}
 M_{t}=\ex \left[ 
 f(W_{1-t}+x)
 \right] \!\big| _{x=\br _{t}} \quad \text{a.s.}
\end{align*}
for any $0\le t\le 1$, which leads to the identity in law: 
\begin{align}\label{;idenlaw}
 \left( \ou _{t}f,\gss{\D }\right) \stackrel{(d)}{=}
 \left( M_{e^{-2t}},\pr \right) \quad \text{for every $t\ge 0$}. 
\end{align}
Our proof of \pref{;pequiv} and \tref{;tunif} utilizes this identity. 

For given $-\infty \le l<r\le \infty $, 
let $\su (t,x),\,(t,x)\in (0,1]\times (l,r)$, be a nonnegative 
$C^{1,2}$-function such that $\su _{x}$ does not vanish. In the remainder 
of this section, we let $f$ be in $C_{b}^{1}(\R ^{\D })$ and suppose 
that $f$ fulfills 
\begin{align}\label{;bounds}
 l<\inf _{x\in \R ^{\D }}f(x)\le \sup _{x\in \R ^{\D }}f(x)<r. 
\end{align}
In the subsequent sections, we take either $-\infty $ or $0$ for $l$ 
and $\infty $ for $r$. In order to develop the process 
$\{ \su (t,M_{t})\} _{0<t\le 1}$ 
by applying It\^o's formula, we use the martingale representation 
for $\{ M_{t}\} _{0\le t\le 1}$. Set a $\D $-dimensional process 
$\dv =\{ \dv _{t}\} _{0\le t\le 1}$ by 
\begin{align}\label{;defv}
 \dv _{t}=\ex \left[ \nabla f(\br _{1-t}+x)\right] \!\big| _{x=\br _{t}}. 
\end{align}

\begin{lem}\label{;lco}
 We have $\pr $-a.s., 
 \begin{align}\label{;msde}
  M_{t}=\ex \left[ f(\br _{1})\right] 
  +\int _{0}^{t}\dv _{s}\cdot d\br _{s}
  \quad \text{for all }0\le t\le 1.  
 \end{align}
\end{lem}

The above lemma is an immediate consequence of the Clark-Ocone 
formula. Because that formula will 
be used again, we provide its rough formulation as introducing 
the necessary notation; we do this in a slightly general situation although 
what we use repeatedly is the simplest case with $m=1$: 
let $F(\br )$ be a functional of $\br $ 
of the form 
\begin{align*}
 F(\br )=\phi \left( \br _{t_{1}},\ldots ,\br _{t_{m}}\right) 
\end{align*}
for some positive integer $m$ and $0\le t_{1},\ldots ,t_{m}\le 1$, 
and for some $\phi \in C^{1}_{b}(\R ^{\D \times m})$. 
We denote by $DF(\br )$ the Malliavin derivative of $F(\br )$, 
which is expressed as 
\begin{align*}
 D_{s}F(\br )=
 \sum _{i=1}^{m}\ind _{[0,t_{i}]}(s)
 \nabla _{x_{i}}\phi \left( \br _{t_{1}},\ldots ,\br _{t_{m}}\right) , 
 \quad 0\le s\le 1. 
\end{align*}
Then the Clark-Ocone formula states that $\pr $-a.s., 
\begin{align}\label{;co}
 \ex \left[ 
 F(\br )|\F _{t}
 \right] 
 =\ex [F(\br )]+\int _{0}^{t}
 \ex \left[ 
 D_{s}F(\br )|\F _{s}
 \right] \cdot d\br _{s}
\end{align}
for all $0\le t\le 1$. For more detailed accounts of the formula, 
see, e.g., \cite[Appendix~E]{ks}, \cite[Section~1.3]{nua}. 

\begin{proof}[Proof of \lref{;lco}]
 Applying \eqref{;co} to $f(\br _{1})$, we have 
 \begin{align*}
 M_{t}=\ex \left[ f(\br _{1})\right] +\int _{0}^{t}
 \ex \left[ 
 \nabla f(\br _{1})|\F_{s}
 \right] \cdot d\br _{s}. 
 \end{align*}
 By the Markov property of $\br $, 
 \begin{align*}
  \ex \left[ 
 \nabla f(\br _{1})|\F_{s}
 \right] 
 =\ex \left[ 
 \nabla f(\br _{1-s}+x)
 \right] \!\big| _{x=\br _{s}} \quad \text{a.s.,}
 \end{align*}
 which ends the proof due to the definition \eqref{;defv} of $\dv $. 
\end{proof}

By \eqref{;msde} and by It\^o's formula, we have 
$\pr $-a.s., 
\begin{align}
 &\su (t,M_{t})-\su (s,M_{s}) \notag \\
 &=\int _{s}^{t}
 \su _{t}(\tau ,M_{\tau })\,d\tau 
 +\int _{s}^{t}\su _{x}(\tau ,M_{\tau })\dv _{\tau }\cdot d\br _{\tau } 
 +\frac{1}{2}\int _{s}^{t}\su _{xx}(\tau ,M_{\tau })
 |\dv _{\tau }|^{2}\,d\tau  \label{;bif}
\end{align}
for all $0<s\le t\le 1$. As $f$ is assumed to satisfy \eqref{;bounds}, 
the stochastic integral above gives rise to a 
true martingale. Therefore, taking the expectation on both sides of 
\eqref{;bif} and differentiating both sides with respect to 
$t$, we obtain the relation 
\begin{align}\label{;diffn}
 \frac{d}{dt}\ex [N_{t}]
 =\ex \left[ 
 \su _{t}(t,M_{t})
 \right] +\frac{1}{2}
 \ex \left[ 
 \su _{xx}(t,M_{t})|\dv _{t}|^{2}
 \right] . 
\end{align}
Here and in what follows, we write 
\begin{align}\label{;defN}
 N_{t}=\su (t,M_{t}), \quad 0<t\le 1. 
\end{align}
We also denote by $\sv (t,\cdot )$ the inverse function 
of $\su (t,x),\,l<x<r$, for each fixed $0<t\le 1$. 

\begin{lem}\label{;lderisv}
 It holds that for any $0<t\le 1$, 
 \begin{align}
  &\su _{x}\!\left( t,\sv (t,\ex [N_{t}])\right) \frac{d}{dt}
  \sv (t,\ex [N_{t}]) \notag \\
  &=-\su _{t}(t,\sv (t,\ex [N_{t}]))
  +\ex \left[ 
  \su _{t}(t,M_{t}) 
  \right] +\frac{1}{2}
  \ex \left[ 
  \su _{xx}(t,M_{t})|\dv _{t}|^{2}
  \right] . \label{;derisv}
 \end{align}
\end{lem}

\begin{proof}
 Observe that due to the relation $x=\su (t,\sv (t,x))$, 
 \begin{align*}
  \sv _{t}(t,x)=-\frac{\su _{t}}{\su _{x}}(t,\sv (t,x)), 
 \end{align*}
 from which we see that 
 \begin{align*}
  \frac{d}{dt}\sv (t,\ex [N_{t}])
  &=\sv _{t}(t,\ex [N_{t}])
  +\sv _{x}(t,\ex [N_{t}])\frac{d}{dt}\ex [N_{t}] \\
  &=-\frac{\su _{t}}{\su _{x}}(t,\sv (t,\ex [N_{t}]))
  +\frac{1}{\su _{x}(t,\sv (t,\ex [N_{t}]))}
  \frac{d}{dt}\ex [N_{t}]. 
 \end{align*}
 Combining this expression with \eqref{;diffn}, 
 we obtain the lemma. 
\end{proof}

In the next lemma, we assume further that for every $0<t\le 1$, 
$\su _{t}$ is twice continuously differentiable with respect to 
the spatial variable $x$. Set 
\begin{align}\label{;defU}
 \U (t,x):=\left\{ 
 \left( 
 \frac{\su _{tx}}{\su _{x}}
 \right) _{x}\frac{1}{\su _{x}}
 \right\} (t,x), \quad (t,x)\in (0,1]\times (l,r). 
\end{align}

\begin{lem}\label{;ldiff}
 It holds that for any $0<t\le 1$, 
 \begin{align}
  &2\su _{x}\!\left( t,\sv (t,\ex [N_{t}])\right) \frac{d}{dt}
  \sv (t,\ex [N_{t}]) \notag \\
  &=\int _{0}^{1}
  \ex \left[ 
  \U \!\left( t,\sv (t,\ex [N_{t}|\F _{s}])\right) 
  \left| \ex [D_{s}N_{t}|\F _{s}]\right| ^{2}
  \right] ds
  +\ex \left[ \su _{xx}(t,M_{t})|\dv _{t}|^{2}\right] . \label{;eqdiff}
 \end{align}
\end{lem}

\begin{proof}
 By the definitions of $\sv $ and $N_{t}$, we may rewrite 
 the integrand of the first term on the right-hand side of 
 \eqref{;diffn} as 
 \begin{align*}%%\label{;1stdiffn}
  \su _{t}(t,M_{t})&=\su _{t}\left( t,\sv (t,\ex [N_{t}|\F _{1}])\right) . 
 \end{align*}
 We apply It\^o's formula to the process 
 $\su _{t}\left( t,\sv (t,\ex [N_{t}|\F _{\tau }])\right) ,\,0\le \tau \le 1$, 
 noting the Clark-Ocone formula (see \eqref{;co}) for 
 $\ex [N_{t}|\F _{\tau }]$: 
 \begin{align*}
  \ex [N_{t}|\F _{\tau }]
  =\ex [N_{t}]+\int _{0}^{\tau }
  \ex [ 
  D_{s}N_{t}|\F _{s}
  ] \cdot d\br _{s}, 
  \quad 0\le \tau \le 1,\ \text{$\pr $-a.s.}
 \end{align*}
 Then it holds that $\pr $-a.s., 
 \begin{align*}
  &\su _{t}\left( t,\sv (t,\ex [N_{t}|\F _{\tau }])\right) \\
  &=\su _{t}\left( t,\sv (t,\ex [N_{t}])\right) 
  +\int _{0}^{\tau }
  \frac{\su _{tx}}{\su _{x}}(t,\sv (t,\ex [N_{t}|\F _{s}]))
  \ex [D_{s}N_{t}|\F _{s}]\cdot d\br _{s}\\
  &\quad +\frac{1}{2}\int _{0}^{\tau }
  \U (t,\sv (t,\ex [N_{t}|\F _{s}]))
  \left| \ex [D_{s}N_{t}|\F _{s}]\right| ^{2}\,ds 
 \end{align*}
 for all $0\le \tau \le 1$. Here we used the fact that 
 $\sv _{x}(t,x)=1/\su _{x}(t,\sv (t,x))$. 
 Taking the expectation on both sides (again the stochastic integral is a 
 true martingale thanks to the boundedness \eqref{;bounds} of $f$) and 
 putting $\tau =1$, we have 
 \begin{align}
  &\ex \left[ 
  \su _{t}(t,M_{t})
  \right] \notag \\
   &=\su _{t}(t,\sv (t,\ex [N_{t}]))
   +\frac{1}{2}\int _{0}^{1}
  \ex \left[ 
  \U \!\left( t,\sv (t,\ex [N_{t}|\F _{s}])\right) 
  \left| \ex [D_{s}N_{t}|\F _{s}]\right| ^{2}
  \right] ds, \label{;coappli}
 \end{align}
 where we used Fubini's theorem for the last term. Plugging 
 \eqref{;coappli} into \eqref{;derisv}, we arrive at the 
 conclusion. 
\end{proof}

%%%%%% New Section %%%%%%
\section{Proof of \pref{;pequiv}}\label{;sprfp}

This section is devoted to the proof of \pref{;pequiv}. We start with 
the proof of the fact that the property that \eqref{;ehc} holds for any 
$f\in \lp{1}{\gss{\D }}$ with $e^{f}\in \lp{1}{\gss{\D }}$, is necessary 
for \eqref{;hc} to hold for any $p>1$ and $f\in \lp{p}{\gss{\D }}$. 

\begin{lem}\label{;lnecessity}
 \eqref{;hc} implies \eqref{;ehc}. 
\end{lem}

\begin{proof}
Fix $t\ge 0$ and let $f\in \lp{1}{\gss{\D }}$ be such that 
$e^{f}\in \lp{1}{\gss{\D }}$. 
Fix $p>1$ arbitrarily and set $g=e^{f/p}$. By Jensen's inequality, 
$\ou _{t}g\ge \exp \left\{ (1/p)\ou _{t}f\right\} $ 
$\gss{\D }$-a.e. The 
hypercontractivity \eqref{;hc} applied to $g$ yields 
$\norm{\ou _{t}g}{q(t)}\le \|e^{f}\|_{1}^{1/p}$. Combining these 
two inequalities, we have  
\begin{align*}
 \norm{\exp \left\{ (1/p)\ou _{t}f\right\} }{q(t)}^{p}
 \le \|e^{f}\|_{1}  
\end{align*}
for any $p>1$. Noting that $q(t)/p\to e^{2t}$ as $p\to \infty $, we let 
$p\to \infty $ on the left-hand side of the above inequality to 
conclude \eqref{;ehc}. 
\end{proof}

We turn to the sufficiency. As mentioned in \sref{;intro}, we use 
the fact (\cite{gro}) that the hypercontractivity \eqref{;hc} 
is equivalent to the logarithmic Sobolev inequality \eqref{;lsi}. 
Thanks to the equivalence, \pref{;pequiv} immediately follows once 
we show the following lemma: 

\begin{lem}\label{;lsufficiency} 
 \eqref{;ehc} implies \eqref{;lsi}. 
\end{lem}

\begin{proof}[Proof of \pref{;pequiv}]
 \lref{;lsufficiency} indicates that \eqref{;ehc} is sufficient 
 for \eqref{;hc} to hold. Combining this fact with \lref{;lnecessity}, 
 we have the proposition. 
\end{proof} 

It remains to prove \lref{;lsufficiency}. We recall the well-known fact that 
taking the derivative of the left-hand side of \eqref{;hc} at $t=0$ leads to 
the logarithmic Sobolev inequality \eqref{;lsi}; the same argument works for 
\eqref{;ehc} as well. 

\begin{proof}[Proof of \lref{;lsufficiency}] 
By density arguments, it suffices to prove \eqref{;lsi} for any 
$f\in \C{\R ^{\D }}$ (for the notation, see the end of \sref{;intro}). 
Pick such an $f$ arbitrarily 
and set $g=2\log f\in C_{b}^{1}(\R ^{\D })$. In view of \eqref{;idenlaw} 
with $g$ replacing $f$ therein, \eqref{;ehc} is restated as 
\begin{align*}
 t\log \ex \left[ 
 \exp \left\{ 
 (1/t)M_{t}(g)
 \right\} 
 \right] \le \log \ex \left[ \exp \{ M_{1}(g)\} \right] 
 \quad \text{for all }0<t\le 1, 
\end{align*}
which in particular entails that 
\begin{align*}%%\label{;deriposi}
 %%\frac{d}{dt}\sv \left( t,\ex [\su (t,M_{t}(g))]\right) 
 \frac{d}{dt}\sv (t,\ex [N_{t}])
 \bigg| _{t=1}\ge 0
\end{align*}
with $N_{t}=\su (t,M_{t}(g))$. Here we set $\su (t,x)=\exp (x/t)$ 
for $(t,x)\in (0,1]\times \R $ with $\sv (t,x)=t\log x$ the inverse function of 
$\su (t,\cdot )$ as in the notation of \sref{;spre}.  
Observe that by choosing $l=-\infty $ and $r=\infty $, the lemmas 
in the previous section are applicable to $g$ and $\su $; in particular, 
we may apply \lref{;lderisv} to see that the last inequality is rewritten as 
\begin{align*}
 -\su _{t}\bigl( 1,\sv (1,\ex [N_{1}])\bigr) 
 +\ex \left[ 
 \su _{t}(1,M_{1}(g))
 \right] +\frac{1}{2}\ex \left[ 
 \su _{xx}(1,M_{1}(g))\left| \nabla g(\br _{1})\right| ^{2}
 \right] \ge 0 
\end{align*}
by the positivity of $\su _{x}$ and by the definition \eqref{;defv} of $\dv $. 
Therefore by the definition of $\su $, we obtain 
\begin{align*}
 \ex \left[ e^{g(\br _{1})}\right] \log \ex \left[ e^{g(\br _{1})}\right] 
 -\ex \left[ 
 g(\br _{1})e^{g(\br _{1})}
 \right] +\frac{1}{2}\ex \left[ 
 e^{g(\br _{1})}\left| \nabla g(\br _{1})\right| ^{2}
 \right] \ge 0. 
\end{align*}
Substituting $g=2\log f$ leads to \eqref{;lsi} and ends the proof. 
\end{proof}

%%In the remainder of the section, we explain how....

%%%%%% New Section %%%%%%
\section{Proof of \tref{;tunif}}\label{;sprft} 
In this section we prove \tref{;tunif}. 
On account of the identity \eqref{;idenlaw} in law, the theorem follows 
once we show the 

\begin{prop}\label{;punif}
 For a function $\cf $ on $(0,\infty )$ satisfying the assumptions 
 in \tref{;tunif}, set 
 \begin{align}
  \su (t,x)=\int _{0}^{x}\cf (y)^{1/t}\,dy, \quad 
  0<t\le 1,\ x>0. \label{;defsu}
 \end{align}
 Then for any nonnegative, measurable function $f$ on $\R ^{\D }$ 
 satisfying \eqref{;intf}, we have 
 \begin{align}\label{;pgenhc}
  \sv \left( t,\ex [\su (t,M_{t}(f))] \right) 
  \le \sv \left( 1,\ex [\su (1,M_{1}(f))]\right) 
  \quad \text{for all }0<t\le 1. 
 \end{align}
 Here for every $0<t\le 1$, we denote by $\sv (t,\cdot )$ the inverse 
 function of $\su (t,x),\,x>0$, as in preceding sections. 
\end{prop}

\begin{proof}[Proof of \tref{;tunif}]
 On noting the identity 
 \begin{align*}
  \uf (t,x)=\su (e^{-2t},x) \quad \text{for all $t\ge 0$ and $x>0$}, 
 \end{align*}
 with a common function $\cf $, the assertion of \tref{;tunif} is immediate 
 from that of \pref{;punif} and the identity \eqref{;idenlaw}. 
\end{proof}

It remains to prove \pref{;punif}. 
To this end, we assume first that $f$ is in $\C{\R ^{\D }}$. 
This assumption will be removed later by density arguments. 
Note that the assumption on $f$ and the definition of $\su $ allow us 
to apply the lemmas in \sref{;spre} by choosing $l=0$ and $r=\infty $; 
in particular, the identity \eqref{;eqdiff} holds true for the 
above pair of $f$ and $\su $, from which we start the proof of the 
proposition. Set 
\begin{align}\label{;defphi}
 \vp (t,x)
 :=-\frac{1}{\U (t,\sv (t,x))}, \quad 0<t\le 1,\ x>0. 
\end{align}

\begin{lem}\label{;lconcave}
 For every $0<t\le 1$, 
 the function $(0,\infty )\ni x\mapsto \vp (t,x)$ 
 is positive and concave. 
\end{lem}

\begin{proof}
 Noting $(\su _{tx}/\su _{x})(t,x)=-(1/t^{2})\log \cf (x)$, 
 we see that $\U $ is expressed as 
 \begin{align*}
  \U (t,x)=-\frac{1}{t^{2}}
  \frac{\cf '(x)}{\cf (x)}\frac{1}{\{ \cf (x)\} ^{1/t}}
 \end{align*}
 by the definition \eqref{;defU} of $\U $. 
 Therefore we have the expression 
 \begin{align}\label{;rewriphi}
  \vp (t,x)
  =t^{2}\frac{\cf (\sv (t,x))}{\cf '(\sv (t,x))}\{ \cf (\sv (t,x))\} ^{1/t}. 
 \end{align}
 The positivity is obvious because $\cf $ and $\cf '$ are positive. 
 To check the concavity, note that $\vp (t,x)$ is both right- and 
 left-differentiable with respect to $x$ since $\cf /\cf '$ is 
 concave and $\sv (t,x)$ is strictly increasing and differentiable 
 with respect to $x$; in fact, if we denote by 
 $(\cf /\cf ')'_{+}$ (resp.\ $(\cf /\cf ')'_{-}$) 
 its right-(resp.\ left-)derivative, then  
 \begin{align}
  \frac{1}{t^{2}}\lim _{h\to 0+}
  \frac{\vp (t,x+h)-\vp (t,x)}{h}
  &=\left( \frac{\cf }{\cf '}\right) '_{+}\!(\sv (t,x))
  \sv _{x}(t,x)\cf (\sv (t,x))^{1/t} \notag \\
  &\quad +\frac{\cf }{\cf '}(\sv (t,x))
  \times \frac{1}{t}\cf (\sv (t,x))^{1/t-1}
  \cf '(\sv (t,x))\sv _{x}(t,x) \notag \\
  &=\left( \frac{\cf }{\cf '}\right) '_{+}\!(\sv (t,x))
  +\frac{1}{t}, \label{;deriphip}
 %%\end{align*}
 \intertext{where the second equality follows from the fact that 
 $\sv _{x}(t,x)=\cf (\sv (t,x))^{-1/t}$; in the same way,}
 %%\begin{align*}
  \frac{1}{t^{2}}\lim _{h\to 0-}
  \frac{\vp (t,x+h)-\vp (t,x)}{h}
  &=\left( \frac{\cf }{\cf '}\right) '_{-}\!(\sv (t,x))
  +\frac{1}{t}. \label{;deriphim}
 \end{align}
 From these identities, the concavity follows because 
 their right-hand sides are nonincreasing in $x$ by the 
 concavity assumption on $\cf /\cf '$. 
\end{proof}

Thanks to the above lemma, we have the following lower bound 
for the expectation in the first term on the right-hand side 
of \eqref{;eqdiff}: 
\begin{lem}\label{;llbd}
 It holds that for every $0<t\le 1$ and $0\le s\le 1$, 
 \begin{align}\label{;eqlbd}
  \ex \left[ 
  \U \!\left( t,\sv (t,\ex [N_{t}|\F _{s}])\right) 
  \left| \ex [D_{s}N_{t}|\F _{s}]\right| ^{2}
  \right] 
  \ge 
  -\ex \left[ 
  \frac{|D_{s}N_{t}|^{2}}{\vp (t,N_{t})}
  \right] . 
 \end{align}
\end{lem}

\begin{proof}
 By the definition \eqref{;defphi} of $\vp $, the left-hand side of 
 \eqref{;eqlbd} is written as 
 \begin{align}\label{;leqlbd}
  -\ex \left[ 
  \frac{\left| \ex [D_{s}N_{t}|\F _{s}]\right| ^{2}}
  {\vp (t,\ex [N_{t}|\F _{s}])}
  \right] .
 \end{align}
 Observe the identity 
 \begin{align}
  &\ex \left[ 
  \vp (t,N_{t})\left| 
  \frac{D_{s}N_{t}}{\vp (t,N_{t})}
  -\frac{\ex [D_{s}N_{t}|\F _{s}]}{\vp (t,\ex [N_{t}|\F _{s}])}
  \right| ^{2}\Bigg| \F _{s}
  \right] \notag \\
  &=\ex \left[ 
  \frac{|D_{s}N_{t}|^{2}}{\vp (t,N_{t})}\bigg| \F _{s}
  \right] 
  -2\frac{\left| \ex [D_{s}N_{t}|\F _{s}]\right| ^{2}}
  {\vp (t,\ex [N_{t}|\F _{s}])}
  +\ex \left[ 
  \vp (t,N_{t})|\F _{s}
  \right] 
  \frac{\left| \ex [D_{s}N_{t}|\F _{s}]\right| ^{2}}
  {\left\{ \vp (t,\ex [N_{t}|\F _{s}])\right\} ^{2}} 
  \quad \text{a.s. }\label{;condexp1}
 \end{align}
 Note that by \lref{;lconcave} and by the conditional Jensen 
 inequality, 
 \begin{align*}
  \ex \left[ 
  \vp (t,N_{t})|\F _{s}
  \right] \le \vp \left( t,\ex [N_{t}|\F _{s}]\right) \quad \text{a.s. }
 \end{align*}
 Plugging this estimate into the third term on the right-hand side of the 
 above identity and using the positivity of $\vp $, we have 
 \begin{align}
  0\le \ex \left[ 
  \frac{|D_{s}N_{t}|^{2}}{\vp (t,N_{t})}\bigg| \F _{s}
  \right] -\frac{\left| \ex [D_{s}N_{t}|\F _{s}]\right| ^{2}}
  {\vp (t,\ex [N_{t}|\F _{s}])} \quad \text{a.s. }\label{;condexp2}
 \end{align}
 Taking the expectation, we see that \eqref{;leqlbd} is bounded from 
 below by the right-hand side of \eqref{;eqlbd}, 
 %%\begin{align*}
 %% -\ex \left[ 
 %% \frac{|D_{s}N_{t}|^{2}}{\vp (t,N_{t})}
 %% \right] , 
 %%\end{align*}
 which ends the proof. 
\end{proof}

We are in a position to prove \pref{;punif}. 
\begin{proof}[Proof of \pref{;punif}]
 First let $f$ be as above, that is, suppose $f\in \C{\R ^{\D }}$. 
 By the definition \eqref{;defN} of $N_{t}$ and by the chain rule for the 
 Malliavin derivative $D$, 
 \begin{align*}
  D_{s}N_{t}=\su _{x}(t,M_{t})D_{s}M_{t}. 
 \end{align*}
 Since $M_{t}$ is written as 
 $
 M_{t}=\ex \left[ 
 f(W_{1-t}+x)
 \right] \!\big| _{x=\br _{t}}
 $, we see that 
 \begin{align*}
  D_{s}M_{t}&=\ind _{[0,t]}(s)\ex \left[ 
  \nabla f(\br _{1-t}+x)
  \right] \!\big| _{x=\br _{t}}
 \intertext{(recall $\nabla f$ is also assumed to be bounded), hence}
  &=\ind _{[0,t]}(s)\dv _{t}
 \end{align*}
 by the definition \eqref{;defv} of $\dv _{t}$. By combining these 
 and by the definition of $N_{t}$, 
 the right-hand side of \eqref{;eqlbd} is expressed as 
 \begin{align*}
  -\ind _{[0,t]}(s)
  \ex \left[ 
  \frac{(\su _{x}(t,M_{t}))^{2}}{\vp (t,\su (t,M_{t})) }|\dv _{t}|^{2}
  \right] . 
 \end{align*}
 By the last expression and by \lsref{;ldiff} and \ref{;llbd}, we have 
 for any $0<t\le 1$, 
 \begin{align}%%\label{;}
  &2\su _{x}\!\left( t,\sv (t,\ex [N_{t}])\right) \frac{d}{dt}
  \sv (t,\ex [N_{t}]) \notag \\
  &\ge 
  \ex \left[ 
  \left\{ 
  -t\frac{(\su _{x}(t,M_{t}))^{2}}{\vp (t,\su (t,M_{t}))}
  +\su _{xx}(t,M_{t})
  \right\} |\dv _{t}|^{2}
  \right] . \label{;bddiff}
 \end{align}
 Recall \eqref{;rewriphi} to note 
 that 
 \begin{align*}
  \vp (t,\su (t,x))=t^{2}\frac{\cf (x)}{\cf '(x)}\{ \cf (x)\} ^{1/t}. 
 \end{align*}
 We also note the expressions of $\su _{x}$ and $\su _{xx}$ 
 in terms of $\cf $: 
 \begin{align*}
  \su _{x}(t,x)=\cf (x)^{1/t}, \quad 
  \su _{xx}(t,x)=\frac{1}{t}\frac{\cf '(x)}{\cf (x)}\{ \cf (x)\} ^{1/t}. 
 \end{align*}
 From these expressions, it follows that for all $0<t\le 1$ and $x>0$, 
 \begin{align*}
  -t\frac{(\su _{x}(t,x))^{2}}{\vp (t,\su (t,x))}
  +\su _{xx}(t,x)&=
  \left( -t\times \frac{1}{t^{2}}+\frac{1}{t}\right) 
  \frac{\cf '(x)}{\cf (x)}\{ \cf (x)\} ^{1/t}\\
  &=0, 
 \end{align*}
 and hence by \eqref{;bddiff}, 
 \begin{align*}
  \frac{d}{dt}\sv (t,\ex [N_{t}])\ge 0 \quad \text{for any }0<t\le 1, 
 \end{align*}
 because $\su _{x}(t,x)$ is positive for all $0<t\le 1$ and $x>0$. 
 Consequently, we have proven \eqref{;pgenhc} when $f\in \C{\R ^{\D }}$. 
 
 The proof of \pref{;punif} is completed by density arguments. To this end, 
 let a measurable function $f$ on $\R ^{\D }$ be such that 
 \begin{align}\label{;bdf}
  \ve \le f\le K \quad \text{$\gss{\D }$-a.e.}
 \end{align}
 for some $0<\ve \le K<\infty $. Then we may choose a sequence 
$\{ f_{n}\} _{n=1}^{\infty }\subset C_{b}^{1}(\R ^{\D })$ such that 
\begin{align}\label{;l1conv}
 \lim _{n\to \infty }\ex \left[ \left| 
 f_{n}(\br _{1})-f(\br _{1})
 \right| \right] =0
\end{align}
and 
$
\ve \le f_{n}(x)\le  K
$ for all $n\ge 1$ and $x\in \R ^{\D }$. 
To see this, we fix $n$ arbitrarily. 
Since any measurable function on $\R ^{\D }$ is approximated by continuous 
functions in the sense of $\gss{\D }$-a.e.\ convergence (see, e.g., 
\cite[Theorem~V.16\,\thetag{a}]{doob}), we may pick a continuous 
function $g_{}$ in such a way that 
\begin{align*}
 \norm{f-g_{}}{1}<n^{-1} \quad 
 \text{and}\quad 
 \ve \le \inf _{x\in \R ^{\D }}g_{}(x)\le \sup _{x\in \R ^{\D }}g_{}(x)
 \le K. 
\end{align*}
Convoluting $g_{}$ with a mollifier on $\R ^{\D }$, we find 
a $\tilde{g}_{}$ in $C_{b}^{1}(\R ^{\D })$ 
(in fact, in $C_{b}^{\infty }(\R ^{\D })$) such that 
\begin{align*}
 \norm{g_{}-\tilde{g}_{}}{1}<n^{-1} \quad 
 \text{and}\quad 
 \ve \le \inf _{x\in \R ^{\D }}\tilde{g}_{}(x)\le 
 \sup _{x\in \R ^{\D }}\tilde{g}_{}(x)\le K. 
\end{align*}
Taking $f_{n}=\tilde{g}_{}$, we have a desired sequence since 
$\norm{f-f_{n}}{1}<2n^{-1}$ by triangular inequality. 
We have already seen that 
\eqref{;pgenhc} holds true for each $f_{n}$:  
\begin{align}\label{;pstrongn}
 \sv \left( 
 t,\ex \left[ 
 \su (t,\ex [f_{n}(\br _{1})|\F _{t}])
 \right] 
 \right) \le 
 \sv \left( 1,\ex \left[ 
 \su (1,f_{n}(\br _{1}))
 \right] \right) . 
\end{align}
By the definition of $\su $ and by the nonnegativity of $\cf '$, 
it holds that for any $0<t\le 1$, 
\begin{align}\label{;lipsu}
 \left| 
 \su (t,x_{1})-\su (t,x_{2})
 \right| 
 \le \cf (K)^{1/t}|x_{1}-x_{2}| 
 \quad \text{for all }x_{1},x_{2}\in [\ve ,K]. 
\end{align} 
Therefore we have the convergence 
\begin{align*}
 &\bigl| 
 \ex \left[ 
 \su (t,\ex [f_{n}(\br _{1})|\F _{t}])
 \right] 
 -\ex \left[ 
 \su (t,\ex [f_{}(\br _{1})|\F _{t}])
 \right] 
 \bigr| \\
 &\le 
 \ex \left[ \bigl| 
 \su (t,\ex [f_{n}(\br _{1})|\F _{t}])
 -\su (t,\ex [f_{}(\br _{1})|\F _{t}])
 \bigr| 
 \right] \\
 &\le \cf (K)^{1/t}\ex \left[ 
 \left| 
 f_{n}(\br _{1})-f(\br _{1})
 \right| 
 \right] \\
 &\xrightarrow[n\to \infty ]{}0
\end{align*}
by \eqref{;l1conv}, which is true for any $0<t\le 1$. 
Here for the third line, we used \eqref{;lipsu} as well as 
the conditional Jensen inequality when $t<1$. Letting 
$n\to \infty $ on both sides of \eqref{;pstrongn}, we have 
\begin{align}\label{;pstrongn2}
 \sv \left( 
 t,\ex \left[ 
 \su (t,\ex [f_{}(\br _{1})|\F _{t}])
 \right] 
 \right) \le 
 \sv \left( 1,\ex \left[ 
 \su (1,f_{}(\br _{1}))
 \right] \right) 
\end{align}
for $f$ satisfying \eqref{;bdf} for some $\ve $ and $K$. 

For a general nonnegative and measurable $f$ satisfying \eqref{;intf}, 
we set 
\begin{align*}
 f_{m,n}:=\min \left\{ 
 \max \left\{ 
 f,1/m
 \right\} ,n
 \right\} 
\end{align*}
for positive integers $m,n$. Then we have \eqref{;pstrongn2} for these 
$f_{m,n}$'s. Appealing to the (conditional) monotone convergence 
theorem, we first let $m\to \infty $ and then $n\to \infty $ to conclude 
the proof. 
\end{proof}

As noted in \sref{;intro}, two choices $x^{p-1}$ ($p>1$) and $e^{x}$ 
for $\cf (x)$ both fulfill \eqref{;cond1}. 
%%Therefore \tref{;tunif} applies to these two functions. 
In the remark below, we explain how the hypercontractivity \eqref{;hc} 
and its variant \eqref{;ehc} are recovered from \tref{;tunif} 
applied to these $\cf $'s and reveal a specific feature of the two 
functions. 

\begin{rem}\label{;runif}
\thetag{1} For $\cf (x)=x^{p-1}$, we have 
\begin{align*}
 \uf (t,x)=\frac{1}{q(t)}x^{q(t)} \quad \text{and} \quad 
 \vf (t,x)=\left\{ q(t)x\right\} ^{1/q(t)}, 
 \quad t\ge 0,\ x>0, 
\end{align*}
with $q(t)=e^{2t}(p-1)+1$, and hence \tref{;tunif} entails that  
\eqref{;hc} holds for every nonnegative $f\in \lp{p}{\gss{\D }}$. 
If $f\in \lp{p}{\gss{\D }}$ is not necessarily nonnegative, 
then noting the fact that $|\ou _{t}f|\le \ou _{t}|f|$ 
$\gss{\D }$-a.e.\ for every $t\ge 0$ (or equivalently, 
$|M_{t}(f)|\le M_{t}(|f|)$ a.s.\ for every $0<t\le 1$ in the formulation 
of the present section), we obtain \eqref{;hc} for any $f\in \lp{p}{\gss{\D }}$. 
As to the choice $\cf (x)=e^{x}$, the corresponding $\uf $ and $\vf $ 
are given respectively by 
\begin{align*}
 \uf (t,x)=e^{-2t}\left\{ 
 \exp (e^{2t}x)-1
 \right\} \quad \text{and} \quad 
 \vf (t,x)=e^{-2t}\log \left( e^{2t}x+1\right) 
\end{align*}
for $t\ge 0$ and $x>0$. Thus for a nonnegative, measurable 
$f$ such that $e^{f}\in \lp{1}{\gss{\D }}$, 
\eqref{;genhc} is restated as 
\begin{align*}
 e^{-2t}\log \norm{\exp \left( e^{2t}\ou _{t}f\right) }{1}
 \le \log \| e^{f}\| _{1}
 \quad \text{for all }t\ge 0, 
\end{align*}
which is nothing but \eqref{;ehc}. For a general $f$ 
satisfying $f\in \lp{1}{\gss{\D }}$ and $e^{f}\in \lp{1}{\gss{\D }}$, set 
$f_{n}=\max \{ f,-n\} $ for each positive integer $n$. Then 
\tref{;tunif} applies to $f_{n}+n$, yielding \eqref{;ehc} with 
$f_{n}$ replacing $f$. Appealing to the monotone convergence theorem, 
we let $n\to \infty $ on both sides and conclude that \eqref{;ehc} 
holds true for any $f\in \lp{1}{\gss{\D }}$ satisfying $e^{f}\in \lp{1}{\gss{\D }}$. 

\noindent 
\thetag{2} In both of the above two cases of $\cf $, 
%%that $\cf (x)=x^{p-1}$ and that $\cf (x)=e^{x}$, 
the corresponding $\vp $ defined by \eqref{;defphi} 
is a linear function in $x$, which fact may be deduced from the 
expressions \eqref{;deriphip} and \eqref{;deriphim}. Therefore in those 
cases, the right-hand side of \eqref{;condexp1} and that of 
\eqref{;condexp2} coincide. In addition, the inequality \eqref{;eqlbd} 
may also be seen by applying the conditional Schwarz inequality to 
$\left| \ex [D_{s}N_{t}|\F _{s}]\right| ^{2}$. We note that 
the Clark-Ocone formula and the conditional Schwarz inequality are both 
main ingredients in a simple derivation \cite{chl} of the logarithmic 
Sobolev inequality over Wiener space. We also remark that 
if $\cf (x)$ satisfies $(\cf /\cf ')''=0$, then it is identical, up to 
affine transformation for $x$, with either $x^\al $ for some 
$\al \neq 0$ or $e^{x}$. 
\end{rem}

The next remark is on the proof of \pref{;punif}. 
\begin{rem}\label{;rpunif}
 \thetag{1} Let $f$ be in $\C{\R ^{\D }}$. 
 In view of the identity \eqref{;condexp1}, if we set 
 a nonnegative function 
 $\Phi \equiv \Phi _{\cf ,f}$ on $(0,1]$ by 
 \begin{align*}
  \Phi (t)=\int _{0}^{1}
  \ex \left[ 
  \vp (t,N_{t})\left| 
  \frac{D_{s}N_{t}}{\vp (t,N_{t})}
  -\frac{\ex [D_{s}N_{t}|\F _{s}]}{\vp (t,\ex [N_{t}|\F _{s}])}
  \right| ^{2}
  \right] ds, 
 \end{align*}
 then what we have in fact shown in the proof is that  
 \begin{align*}
  2\su _{x}\!\left( t,\sv (t,\ex [N_{t}])\right) \frac{d}{dt}
  \sv (t,\ex [N_{t}])\ge \Phi (t)
 \end{align*}
 for any $0<t\le 1$. Here equality holds if 
 $\vp $ is linear in the spatial variable, which is the case 
 when $\cf (x)$ is $x^{p-1}$ for some $p>1$, as well as when 
 $\cf (x)$ is $e^{x}$ as noted in \rref{;runif}\,\thetag{2}. 
 In the former case, by dividing both sides of the equality by 
 the quantity 
 \begin{align*}
  \sv (t,\ex [N_{t}])\su _{x}\!\left( t,\sv (t,\ex [N_{t}])\right) 
  =\ex \bigl[ 
  \left\{ 
  M_{t}(f)
  \right\} ^{(p-1)/t+1}
  \bigr] , 
 \end{align*}
 the following identity is easily deduced on account of 
 \eqref{;idenlaw}: 
 \begin{align*}
  \norm{\ou _{t}f}{q(t)}=\norm{f}{p}
  \exp \left\{ 
  -\int _{0}^{t}\frac{e^{-2\tau }}
  {\norm{\ou _{\tau }f}{q(\tau )}^{q(\tau )}}
  \Phi (e^{-2\tau })\,d\tau 
  \right\} & && \text{for all }t\ge 0; 
\intertext{in the latter case 
$\cf (x)=e^{x}$, a similar identity also holds: }
  \norm{\exp \left( \ou _{t}f\right) }{e^{2t}}
  =\| e^{f}\| _{1}
  \exp \left\{ 
  -\int _{0}^{t}\frac{e^{-2\tau }}
  {\norm{\exp \left( \ou _{\tau }f\right) }{e^{2\tau }}^{e^{2\tau }}}
  \Phi (e^{-2\tau })\,d\tau 
  \right\} & && \text{for all }t\ge 0. 
 \end{align*}
 
 \noindent 
 \thetag{2} Let $\su (t,x), 0<t\le 1,\,x>0$, be a generic, 
 positive and 
 smooth function with $\su _{x}>0$. The derivation of \pref{;punif} 
 hinges upon the fact that we are able to solve the following 
 pair of equations in $(0,1)\times (0,\infty )$: 
 \begin{align*}
 \begin{cases}
  \psi (t,\su (t,x))\U (t,x)=-\dfrac{1}{t}, \\
  \psi (t,\su (t,x))\dfrac{\su _{xx}(t,x)}{\{ \su _{x}(t,x)\} ^{2}}
  =1, 
 \end{cases}
 \end{align*}
 where $\U $ is defined by \eqref{;defU} and $\psi $ is also an unknown, 
 positive function such that $\psi (t,\cdot )$ is required to be concave 
 for every $0<t\le 1$. Indeed, by these equations, $\su $ must satisfy 
 \begin{align*}
  t\U (t,x)+\frac{\su _{xx}(t,x)}{\{ \su _{x}(t,x)\} ^{2}}=0 
  \quad \text{in }(0,1)\times (0,\infty ), 
 \end{align*}
 which equation is rephrased as 
 \begin{align*}
  \frac{\partial ^{2}}{\partial x\partial t}
  \left\{ 
  t\log \su _{x}(t,x)
  \right\} =0. 
 \end{align*}
 Then $\su _{x}$ is expressed, up to multiple of a positive 
 function in $t$, as 
 \begin{align*}
  %%\su _{x}(t,x)=\exp \left( \frac{C(x)}{t}\right) 
  \su _{x}(t,x)=e^{C(x)/t}
 \end{align*}
 with $C$ a differentiable function in $x$; the associated $\psi $ 
 is given by the product of a positive function in $t$ and 
 $(1/C'(\sv (t,x))\exp \{ C(\sv (t,x))/t\} $, 
 which is found to be concave in $x$ when $C'>0$ and $1/C'$ is concave. 
 Here $\sv $ is the inverse function of $\su $ in the spatial variable 
 as in the notation of \sref{;spre}. 
\end{rem} 

We give examples of functions $\cf $ satisfying \eqref{;cond1} 
and show consequences of \tref{;tunif} corresponding to 
them.  
\begin{exm}\label{;cexample}
\thetag{1} For two exponents $\al ,\beta $ satisfying 
$\al +\beta \ge 1$ and $0<\beta \le 1$, we take 
\begin{align*}
 \cf (x)=x^{\al +\beta -1}\exp \left( x^{\beta }\right) , 
 \quad x>0. 
\end{align*}
If we write $\rho =\al +\beta -1$, then 
\begin{align*}
 \cf '(x)=\left( \rho x^{\rho -1}
 +\beta x^{\rho +\beta -1}\right)
 \exp \left( x^{\beta }\right) ,  
\end{align*}
which is positive for all $x>0$ when $\rho \ge 0$ and 
$\beta >0$. Noting 
\begin{align*}
 \frac{\cf }{\cf '}(x)=
 \frac{x}{\rho +\beta x^{\beta }}, 
\end{align*}
we find that 
\begin{align*}
 \left\{ 
 \frac{\cf }{\cf '}(x)
 \right\} '
 =\frac{\rho +(1-\beta )y}{(\rho +y)^{2}}
 \bigg| _{y=\beta x^{\beta }}. 
\end{align*}
The function $\{ \rho +(1-\beta )y\} /(\rho +y)^{2}$ in $y>0$ 
is nonincreasing when $\beta \le 1$ and $\rho (1+\beta )\ge 0$, 
and hence under the condition imposed on $\al $ and $\beta $, 
the above $\cf $ satisfies \eqref{;cond1}. Observe that by 
L'H\^opital's rule, the corresponding $\uf $ admits the asymptotics 
\begin{align*}
 \uf (t,x)
 \sim \frac{e^{-2t}}{\beta }
 x^{e^{2t}(\al +\beta -1)-\beta +1}
 \exp \left( e^{2t}x^{\beta }\right) 
\end{align*}
as $x\to \infty $ for every $t\ge 0$. Here and 
below, the notation $\sim $ indicates that 
the ratio of both sides in the equation converges to 
$1$ when $x\to \infty $. As a consequence, 
we deduce from \tref{;tunif} that the following implication 
is true: for any nonnegative, measurable function $f$ 
on $\R ^{\D }$, 
\begin{align*}
 f^{\al }\exp \left( f^{\beta }\right) \in \lp{1}{\gss{\D }} 
 \ \Rightarrow \ 
 (\ou _{t}f)^{e^{2t}(\al +\beta -1)-\beta +1}
 \exp \left\{ 
 e^{2t}(\ou _{t}f)^{\beta }
 \right\} \in \lp{1}{\gss{\D }}, 
 \ \forall t\ge 0. 
\end{align*}

\noindent 
\thetag{2} For positive reals $\al $ and $\beta $, take 
\begin{align*}
 \cf (x)=\frac{(x+a)^{\al }}{\log ^{\beta }(x+a)}, \quad x>0, 
\end{align*}
where $a$ is a constant satisfying $a\ge e^{2+\rho }$ with 
$\rho :=\beta /\al $. 
Then 
\begin{align*}
 \cf '(x)=\frac{\al (x+a)^{\al -1}}{\log ^{\beta +1}(x+a)}
 \left\{ 
 \log (x+a)-\rho 
 \right\} >0 \quad \text{for all }x>0, 
\end{align*}
and 
\begin{align*}
 \frac{\cf }{\cf'}(x)
 =\frac{x+a}{\al }+\frac{\rho }{\al }\cdot 
 \frac{x+a}{\log (x+a)-\rho }
\end{align*}
is concave on $(0,\infty )$. Indeed, 
\begin{align*}
 \left\{ \frac{x+a}{\log (x+a)-\rho }\right\} '
 =\frac{y-1}{y^{2}}\bigg| _{y=\log (x+a)-\rho }, \quad x>0, 
\end{align*}
the function $(y-1)/y^{2}$ being decreasing on $(2,\infty )$. 
Therefore \tref{;tunif} applies to the above 
choice of $\cf $ as well. Since the corresponding $\uf $ 
admits the asymptotics 
\begin{align*}
 \uf (t,x)\sim 
 \frac{1}{e^{2t}\al +1}\cdot 
 \frac{x^{e^{2t}\al +1}}{\log ^{e^{2t}\beta }x} 
 \quad \text{as }x\to \infty 
\end{align*}
for every $t\ge 0$, there holds the following implication: 
for any nonnegative, measurable function $f$ on $\R ^{\D }$, 
\begin{align*}
 \frac{f^{\al +1}}{\log ^{\beta }(f+b)}\in \lp{1}{\gss{\D }}\ 
 \Rightarrow \ 
 \frac{(\ou _{t}f)^{e^{2t}\al +1}}
 {\log ^{e^{2t}\beta }(\ou _{t}f+b)}\in \lp{1}{\gss{\D }}, \ \forall t\ge 0. 
\end{align*}
Here $b$ is any constant greater than $1$. 
\end{exm}

We end this section by providing a generalization of the logarithmic Sobolev inequality \eqref{;lsi} as a corollary to  \tref{;tunif}. 

\begin{cor}\label{;cglsi}
 For a function $\cf :(0,\infty )\to (0,\infty )$ satisfying the 
 assumptions in \tref{;tunif}, set 
 \begin{align*}
  \G (x)=\int _{0}^{x}\cf (y)\,dy && \text{and} &&  
  \h (x)=\int _{0}^{x}\cf (y)\log \cf (y)\,dy 
  %%\text{for }x>0. 
 \end{align*}
 for $x>0$. Then for any $f\in \C{\R ^{\D }}$, 
 we have 
 \begin{align}\label{;glsi}
  \int _{\R ^{\D }}\h (f)\,d\gss{\D }
  \le \frac{1}{2}\int _{\R ^{\D }}\cf '(f)\left| \nabla f\right| ^{2}d\gss{\D }
  +H\circ G^{-1}\left( \norm{\G (f)}{1}\right) . 
 \end{align}
 Here $\G ^{-1}$ is the inverse function of $\G $. 
\end{cor}

\begin{rem}\label{;nolabel}
 It is plausible that the inequality \eqref{;glsi} would be extended to 
 the class of functions $f$ for which every term in the inequality 
 makes sense, however, we do not pursue it here. 
\end{rem}

The proof of \cref{;cglsi} proceeds along the same lines as in the 
proof of \lref{;lsufficiency}. 
\begin{proof}[Proof of \cref{;cglsi}]
 For the proof, we use \pref{;punif}, the equivalent statement of 
 \tref{;tunif}. We see from \eqref{;pgenhc} that 
 \begin{align*}
  \frac{d}{dt}\sv \left( t,\ex \left[ \su (t,M_{t}(f))\right] \right) 
  \bigg| _{t=1}\ge 0, 
 \end{align*}
 which is rewritten, by \lref{;lderisv}, as 
 \begin{align*}
  -\su _{t}(1,\sv (1,\ex [N_{1}]))
 +\ex \left[ 
 \su _{t}(1,M_{1}(f)) 
 \right] +\frac{1}{2}
 \ex \left[ 
 \su _{xx}(1,M_{1}(f))\left| \nabla f(\br _{1})\right| ^{2}
 \right] \ge 0
 \end{align*}
 due to the definition \eqref{;defv} of $\dv $ and the positivity of 
 $\su _{x}$. The last inequality is nothing but 
 \eqref{;glsi} because of the relations 
 $\su _{t}(1,\cdot )=\h $, $\su (1,\cdot )=\G $ and 
 $\sv (1,\cdot )=\G ^{-1}$ by the definitions of 
 $\su $, $\h $ and $\G $, as well as because of the identities 
 $M_{1}(f)=f(\br _{1})$ and $N_{1}=\su (1,f(\br _{1}))=\G (f(\br _{1}))$. 
\end{proof}

If we choose $x^{p-1}$ ($p>1$) or $e^{x}$ for $\cf (x)$ in \eqref{;glsi}, 
then \eqref{;lsi} is recovered; details are left to the reader. 

%%%%%% New Section %%%%%%
\section{On the reverse hypercontractivity}\label{;srhc}
This section concerns the reverse hypercontractivity of the 
Ornstein-Uhlenbeck semigroup $\ou $. We begin with the following 
proposition, which is proven by modifying slightly the proof of 
\tref{;punif}. 

\begin{prop}\label{;pgenrhc}
 Let a positive function $\cf $ on $(0,\infty )$ be in 
 $C^{1}((0,\infty ))$ and satisfy 
 \begin{align}\label{;cond2}\tag{C$'$}
  \text{$\cf '<0$, 
  $\dfrac{\cf}{\cf '}$ is convex on $(0,\infty )$ and 
  $\lim \limits _{x\to 0+}\cf (x)<\infty $.}
 \end{align}
 We set the function $\uf (t,x),\,t\ge 0,x>0$, by \eqref{;defuf}: 
 \begin{align*}
  \uf (t,x)=\int _{0}^{x}\cf (y)^{e^{2t}}\,dy. 
  %%, \quad t\ge 0,\ x>0. 
 \end{align*}
 Then for any $f\in \C{\R ^{\D }}$, 
 we have 
 \begin{align}\label{;genrhc}
  \vf \left( t,\norm{\uf (t,\ou _{t}f)}{1}\right) 
  \ge \vf \left( 0,\norm{\uf (0,f)}{1}\right) \quad \text{for all }t\ge 0. 
 \end{align}
 Here for every $t\ge 0$, the function 
 $\vf (t,\cdot )$ is the inverse function of 
 $\uf (t,x),\,x>0$. 
\end{prop}

In \eqref{;cond2}, the last condition is put to ensure the finiteness of 
$\uf $. By the identity \eqref{;idenlaw} in law, the assertion of the 
proposition  is restated as 
 \begin{align}\label{;egenrhc}
  \sv \left( t,\ex [\su (t,M_{t}(f))] \right) 
  \ge \sv \left( 1,\ex [\su (1,M_{1}(f))]\right) 
  \quad \text{for all }0<t\le 1, 
 \end{align}
 for each $f\in \C{\R ^{\D }}$. 
 Here $\su $ is defined by \eqref{;defsu} with $\cf $ satisfying 
 \eqref{;cond2} and $\sv $ denotes the inverse function 
 of $\su $ in the spatial variable as before. 
 
\begin{proof}[Proof of \pref{;pgenrhc}]
 Observe that with taking 
 $l=0$ and $r=\infty $, the lemmas in \sref{;spre} apply to the above 
 choice of $f$ and $\su $, and hence \lref{;ldiff} is valid. We shall see that 
 the right-hand side of \eqref{;eqdiff} in that lemma 
 does not exceed $0$, which entails 
 \begin{align}\label{;derineg}
  \frac{d}{dt}\sv (t,\ex [N_{t}])\le 0 \quad \text{for any }0<t\le 1, 
 \end{align}
 by the positivity of $\su _{x}$. To this end, recall the definition 
 \eqref{;defphi} of $\vp $; its expression \eqref{;rewriphi} in terms of 
 $\cf $ is valid in the present case as well, and in particular 
 it reveals, by \eqref{;cond2} and by repeating the same argument 
 as in the proof of \lref{;lconcave}, that $\vp $ is negative and is 
 a convex function in $x$ for every $0<t\le 1$. Remembering 
 the identity 
 \eqref{;condexp1} in the proof of \lref{;llbd} and noting that 
 \begin{align*}
  \ex \left[ 
  \vp (t,N_{t})|\F _{s}
  \right] \ge \vp \left( t,\ex [N_{t}|\F _{s}]\right) \quad \text{a.s. }
 \end{align*}
 by the conditional Jensen inequality, we obtain 
 \begin{align*}
  0\ge \ex \left[ 
  \frac{|D_{s}N_{t}|^{2}}{\vp (t,N_{t})}\bigg| \F _{s}
  \right] -\frac{\left| \ex [D_{s}N_{t}|\F _{s}]\right| ^{2}}
  {\vp (t,\ex [N_{t}|\F _{s}])} \quad \text{a.s. }
 \end{align*}
 in place of \eqref{;condexp2}, due to the negativity of $\vp $. 
 By the definition of $\vp $, the last inequality leads to \eqref{;eqlbd} 
 with the reversed inequality sign. The rest of the proof for 
 \eqref{;derineg} proceeds along the same lines as in the first part of 
 the proof of \pref{;punif}. Therefore \eqref{;egenrhc} follows and 
 we obtain the proposition. 
\end{proof}

The next proposition shows that \eqref{;genrhc} unifies two properties of 
$\ou $, the reverse hypercontractivity \eqref{;rhc} and the 
exponential variant \eqref{;ehc} of the hypercontractivity. 

\begin{prop}\label{;punif2}
 The property \eqref{;genrhc} implies \eqref{;rhc} and 
 \eqref{;ehc}. 
\end{prop}

In view of \eqref{;idenlaw}, in order to prove the assertion, 
it suffices to show that we may derive from \eqref{;egenrhc} 
the following: given $\al >0$, 
 \begin{align}\label{;pstrong}
  \ex \left[ 
  \left( 
  \frac{1}{M_{t}(f)}
  \right) ^{\rho _{\al }(t)}
  \right] ^{1/\rho _{\al }(t)}
  \le \ex \left[ 
  \left( \frac{1}{M_{1}(f)}\right) ^{\al }
  \right] ^{1/\al } \quad \text{for all }0<t\le 1, 
 \end{align}
 for every a.e.\ nonnegative $f\in \lp{1}{\gss{\D }}$ 
 satisfying $1/f\in \lp{\al }{\gss{\D }}$, as well as  
 \begin{align}\label{;exp0}
 \ex \left[ 
 \exp \left\{ 
 (1/t)M_{t}(f)
 \right\} 
 \right] ^t\le \ex \left[ 
 \exp \left\{ 
 M_{1}(f)
 \right\} 
 \right] \quad \text{for all }0<t\le 1, 
\end{align}
for every $f\in \lp{1}{\gss{\D }}$ satisfying $e^{f}\in \lp{1}{\gss{\D }}$. 
Here we set $\rho _{\al }(t)=(\al +1)/t-1$ in \eqref{;pstrong}.  

\begin{proof}[Proof of \pref{;punif2}]
 We begin with the proof of \eqref{;pstrong}. 
 For this purpose, let $f$ be in $\C{\R ^{\D }}$ first. We pick a 
 constant $\kp >0$ in such 
 a way that $\kp <\inf \limits_{x\in \R ^{\D }}f(x)$ and take 
 \begin{align*}
  \cf (x)=\frac{1}{(x+\kp )^{\al +1}},\quad x>0. 
 \end{align*}
 Then the condition \eqref{;cond2} is fulfilled; in fact, 
 $(\cf /\cf ')(x)=-(x+\kp )/(\al +1)$, and hence $(\cf /\cf ')''=0$. 
 Therefore we have \eqref{;egenrhc} with this choice of $\cf $. The 
 corresponding $\su $ and $\sv $ are given respectively by 
 \begin{align*}
  \su (t,x)&=\left\{ 
  \kp ^{-\rho _{\al }(t)}-(x+\kp )^{-\rho _{\al }(t)}
  \right\} /\rho _{\al }(t), && 0<t\le 1,\,x>0, \\
  \sv (t,x)&=\left\{ 
  \kp ^{-\rho _{\al }(t)}-\rho _{\al }(t)x
  \right\} ^{-1/\rho _{\al }(t)}-\kp , && 
  0<t\le 1,\,0<x<\kp ^{-\rho _{\al }(t)}/\rho _{\al }(t). 
 \end{align*}
 From these expressions, we see that applying \eqref{;egenrhc} to 
 $f-\kp $ yields \eqref{;pstrong} for every $f\in \C{\R ^{\D }}$. 
 
 Next we assume that $f$ is in $\lp{1}{\gss{\D }}$ and satisfies 
 \begin{align*}
  \ve :=\essinf _{x\in \R ^{\D }}f(x)>0. 
 \end{align*}
 Then we may choose a sequence 
 $\{ f_{n}\} _{n=1}^{\infty }\subset C_{b}^{1}(\R ^{\D })$ in such a 
 way that 
 \begin{align}\label{;l1conv2}
  \lim _{n\to \infty }\ex \left[ \left| 
  f_{n}(\br _{1})-f(\br _{1})
  \right| \right] =0
 \end{align}
 and $\inf \limits_{x\in \R ^{\D }}f_{n}(x)\ge \ve $ for all 
 $n\ge 1$ (see, e.g., the middle part of the proof of 
 \pref{;punif}). We have seen in the previous step that 
 \eqref{;pstrong} holds true for each $f_{n}$:  
 \begin{align}\label{;pstrongn3}
 \ex \left[ 
 \left( 
 \frac{1}{\ex [f_{n}(\br _{1})|\F _{t}]}
 \right) ^{\rho _{\al }(t)}
 \right] ^{1/\rho _{\al }(t)}
 \le \ex \left[ 
 \left( \frac{1}{f_{n}(\br _{1})}\right) ^{\al }
 \right] ^{1/\al }
\end{align}
for all $0<t\le 1$. By the inequality 
$
\left| (1/x_{1})^{\al }-(1/x_{2})^{\al }\right| 
\le \al |x_{1}-x_{2}|/\ve ^{\al +1}
$ for $x_{1},x_{2}\ge \ve $, and by \eqref{;l1conv2}, 
we have the convergence of expectations 
\[ 
 \lim _{n\to \infty }
 \ex \left[ 
 \left( \frac{1}{f_{n}(\br _{1})}\right) ^{\al }
 \right] =
 \ex \left[ 
 \left( \frac{1}{f(\br _{1})}\right) ^{\al }
 \right]  
\] 
as to the right-hand side of \eqref{;pstrongn3}. The same 
reasoning combined with the conditional Jensen inequality 
yields the convergence as $n\to \infty $ of the left-hand 
side of \eqref{;pstrongn3} to the expression 
with $f_{n}$ replaced by $f$.  Hence we have obtained 
\eqref{;pstrong} for $f\in \lp{1}{\gss{\D }}$ with 
$\essinf \limits_{x\in \R ^{\D }}f(x)>0$. 

Finally, let $f$ be an a.e.\ positive function in $\lp{1}{\gss{\D }}$ 
satisfying $1/f\in \lp{\al }{\gss{\D }}$. For every positive integer $n$, set 
\begin{align*}
 f_{n}:=\max \left\{ 
 f,1/n
 \right\} . 
\end{align*}
Then we have \eqref{;pstrongn3} for these $f_{n}$'s as well. 
Since each $f_{n}$ is nonnegative and dominated by 
the integrable function $f+1$, it holds that 
\begin{align*}
 \lim _{n\to \infty }
 \ex [f_{n}(\br _{1})|\F _{t}]=\ex [f(\br _{1})|\F _{t}] 
 \quad \text{a.s.}
\end{align*}
by the conditional dominated convergence theorem. Letting 
$n\to \infty $ on both sides of \eqref{;pstrongn3}, we 
appeal to the monotone convergence theorem to conclude 
the validity of \eqref{;pstrong} for any a.e.\ positive 
$f \in \lp{1}{\gss{\D }}$. 

We turn to the proof of \eqref{;exp0}. 
First we pick $f\in C^{1}_{b}(\R ^{\D })$ 
and let $\kp \ge 0$ be a constant such that 
$\inf \limits _{x\in \R ^{\D }}\left\{ -f(x)\right\} >-\kp $. We may 
take $\cf (x)=e^{-(x-\kp )},\,x>0$, in \pref{;pgenrhc}; indeed, 
$\cf /\cf '$ is identically equal to $-1$ and thus the condition 
\eqref{;cond2} is fulfilled. Then the inequality \eqref{;egenrhc} 
applied to $-f+\kp \in \C{\R ^{\D }}$ entails \eqref{;exp0} for 
every $f\in C^{1}_{b}(\R ^{\D })$, due to the expressions of the 
corresponding $\su $ and $\sv $: 
\begin{align*}
 \su (t,x)&=t\left\{ e^{\kp /t}-e^{-(x-\kp )/t}\right\} , && 
 0<t\le 1,\,x>0, \\
 \sv (t,x)&=-t\log \left( 1-e^{-\kp /t}x/t\right) , && 
 0<t\le 1,\,0<x<te^{\kp /t}. 
\end{align*}

 Next let $f$ be in $\lp{1}{\gss{\D }}$ and satisfy 
 $e^{f}\in \lp{1}{\gss{\D }}$. We write $\{ f_{n}\} _{n=1}^{\infty }$ for 
 a sequence in $C_{b}^{1}(\R ^{\D })$ that approximates 
 $f$ in $\lp{1}{\gss{\D }}$. If we suppose that 
 $K:=\esssup \limits_{x\in \R ^{\D }}f(x)<\infty $, 
 then we may take the above sequence in such a way that 
 \begin{align*}
  \sup _{x\in \R ^{\D }}f_{n}(x)\le K 
 \end{align*}
 for all $n$ (cf.\ Proof of \pref{;punif}). 
 We have observed in the previous step that for each $f_{n}$, 
 \begin{align}\label{;prfexp1}
  \ex \left[ 
  \exp \left\{ 
  (1/t)\ex \left[ 
  f_{n}(\br _{1})|\F _{t}
  \right] 
  \right\} 
  \right] ^{t}
  \le \ex \left[ 
  e^{f_{n}(\br _{1})}
  \right] 
 \end{align}
 holds for all $0<t\le 1$. By noting that 
 $\left| e^{x_{1}}-e^{x_{2}}\right| \le e^{K}|x_{1}-x_{2}|$ 
 for $x_{1},x_{2}\le K$, 
 \begin{align*}
  \left| 
  \ex \left[ 
  e^{f_{n}(\br _{1})}
  \right] 
  -\ex \left[ 
  e^{f_{}(\br _{1})}
  \right] 
  \right| 
  \le 
  e^{K}\ex \left[ 
  \left| 
  f_{n}(\br _{1})-f(\br _{1})
  \right| 
  \right] \xrightarrow[n\to \infty ]{}0
 \end{align*}
 since $f_{n}$ approximates $f$ in $\lp{1}{\gss{\D }}$. The same 
 reasoning with $K/t$ replacing $K$, yields the convergence of 
 the expectation in the left-hand side of \eqref{;prfexp1}, 
 and hence we have 
 \begin{align*}
  \ex \left[ 
  \exp \left\{ 
  (1/t)\ex \left[ 
  f_{}(\br _{1})|\F _{t}
  \right] 
  \right\} 
  \right] ^{t}
  \le \ex \left[ 
  e^{f_{}(\br _{1})}
  \right] 
 \end{align*}
 when $f$ is bounded from above. For a general $f$ satisfying the 
 assumption, we truncate $f$ from above and use the monotone 
 convergence theorem and its conditional version (or the 
 conditional dominated convergence theorem) to reach 
 the conclusion. 
\end{proof}

\begin{rem}
 As noted in \cite[p.~274]{bgl}, the reverse hypercontractivity 
 \eqref{;rhc} was firstly observed by Borell and Janson 
 \cite{bj}${}^{\dagger}$\footnote{${}^{\dagger}$There seems to 
 be some confusion in the literature. 
 In \cite{bgl}, this paper is referred to as [88] in the bibliography, 
 which is found to be Borell's single-authored paper 
 entitled ``Positivity 
 improving operators and hypercontractivity'' (Math.\ Z.\ {\bf 180} (1982), 
 no.\ 2, 225--234); it is true that in this Borell's paper, the reverse 
 hypercontractivity is presented with its different proof than the original 
 one, however, the paper cited there for the original proof supposedly needs 
 to be corrected as the reference \cite{bj} in the present paper. } in the 
 name of ``converse hypercontractivity.'' It is also noted 
 in \cite[Remark~5.2.4]{bgl} that in the same way of 
 the hypercontractivity \eqref{;hc} yielding the logarithmic 
 Sobolev inequality \eqref{;lsi} and vice versa, 
 the reverse hypercontractivity is seen to be equivalent 
 to \eqref{;lsi} as well. 
\end{rem}

If we let $\al \downarrow 0$ in \eqref{;rhc}, we immediately obtain the 
following claim, which is of interest itself and which, to our knowledge, 
has not ever been stated in a clear manner. 

\begin{prop}\label{;ctmain}
 Suppose $f:\R ^{\D }\to \R $ is positive $\gss{\D }$-a.e.\ and 
 in $\lp{1}{\gss{\D }}$. 
 If $f$ satisfies 
 \begin{align*}
  \int _{\R ^{\D }}
  \log ^{+}( 1/f)\,d\gss{\D }
  <\infty ,
 \end{align*}
 then for every $t>0$, $1/\ou _{t}f$ is in $\lp{e^{2t}-1}{\gss{\D }}$; 
 in fact, it holds that 
 \begin{align*}
  \norm{1/\ou _{t}f}{e^{2t}-1}\le 
  \exp \left( -\int _{\R ^{\D }}\log f\,d\gss{\D }\right) 
 \end{align*}
 for all $t>0$. 
 Here $\log ^{+}x:=\max \left\{ \log x,0\right\} ,\,x>0$. 
\end{prop}

\begin{proof}
 By \pref{;punif2} and by the identity \eqref{;idenlaw} 
 in law, we may start the proof from \eqref{;pstrong} when 
 $f\in \lp{1}{\gss{\D }}$ is bounded away from $0$: 
 $\essinf \limits_{x\in \R ^{\D }}f(x)>0$. Then 
 by the boundedness of 
 $1/f$, it is easily seen that as $\al \downarrow 0$, the left-hand side of 
 \eqref{;pstrong} converges to the expression 
 with $\rho _{\al }(t)$ replaced by 
 $\rho _{0}(t):=1/t-1$ for every $0<t<1$. 
(In fact, what we actually need for the proof is a simple 
fact that the above-mentioned expression does not exceed the 
left-hand side of \eqref{;pstrong}.) As for the right-hand side of 
\eqref{;pstrong}, we rewrite it into 
\begin{align}\label{;rhs}
 \exp \left\{ 
 \frac{1}{\al }\log \ex \left[ 
 \left( M_{1}(f)\right) ^{-\al }
 \right] 
 \right\} .
\end{align}
Recall $M_{1}(f)=f(\br _{1})$. Observe that for any open 
and bounded interval $I\subset (0,\infty )$, the random variable 
\begin{align*}
 \sup _{\al \in I}\left( M_{1}(f)\right) ^{-\al }
 \left| 
 \log M_{1}(f)
 \right| 
\end{align*}
is integrable thanks to the boundedness of $1/f$ and the 
condition $f\in \lp{1}{\gss{\D }}$. This observation entails that 
on $(0,\infty )$, there holds the equality 
\begin{align*}
 \frac{d}{d\al }\ex \left[ 
 \left( 
 M_{1}(f)
 \right) ^{-\al }
 \right] 
 =-\ex \left[ 
 \left( 
 M_{1}(f)
 \right) ^{-\al }\log M_{1}(f)
 \right] 
\end{align*}
whose right-hand side converges as $\al \downarrow 0$ to 
$-\ex \left[ \log M_{1}(f)\right] $ by the dominated convergence 
theorem. Therefore applying L'H\^opital's rule, we have 
\begin{align*}
 \lim _{\al \downarrow 0}\frac{1}{\al }
 \log \ex \left[ 
 \left( M_{1}(f)\right) ^{-\al }
 \right] 
 =-\ex \left[ \log M_{1}(f)\right] ,
\end{align*}
and hence \eqref{;rhs}, or the right-hand side of \eqref{;pstrong}, 
converges to $\exp \left\{ -\ex \left[ \log M_{1}(f)\right] \right\} $ 
as $\al \downarrow 0$. Consequently, we obtain 
\begin{align}\label{;pcstrong}
  \ex \left[ 
  \left( M_{t}(f)\right) ^{-\rho _{0}(t)}
  \right] ^{1/\rho _{0}(t)}
  \le \exp \left\{ 
  -\ex \left[ 
  \log M_{1}(f)
  \right] 
  \right\} 
 \end{align}
for any $f\in \lp{1}{\gss{\D }}$ which is bounded away from $0$. 

For a general $f$ satisfying the assumption, we set 
$f_{n}=\max \left\{ f,1/n\right\} $ for each positive integer $n$ 
as in the last step of the proof of \eqref{;pstrong}. 
By the same reasoning as used there, 
\begin{align*}
 \lim _{n\to \infty }\ex \left[ 
 \left( 
 M_{t}(f_{n})
 \right) ^{-\rho _{0}(t)}
 \right] 
 =\ex \left[ 
 \left( 
 M_{t}(f_{})
 \right) ^{-\rho _{0}(t)}
 \right] . 
\end{align*}
We also have the convergence 
\begin{align*}
 \lim _{n\to \infty }\ex \left[ 
 \log M_{1}(f_{n})
 \right] 
 =\ex \left[ 
 \log M_{1}(f)
 \right] 
\end{align*}
by the monotone convergence theorem because of 
the domination 
$\log f_{n}\le f_{n}-1\le f$ for all $n$. Since \eqref{;pcstrong} 
holds for any $f_{n}$ replacing $f$ as has already been observed, 
letting $n\to \infty $ on both sides leads to the desired conclusion 
thanks to \eqref{;idenlaw}. 
\end{proof}

We end this paper with a remark on the last proposition. 
\begin{rem} 
 \pref{;ctmain} may also be proven by taking 
 \begin{align*}
  \cf (x)=\frac{1}{x+\kp }, \quad x>0, 
 \end{align*}
 in \eqref{;egenrhc} with $\kp $ a positive constant, and by 
 repeating the same argument as in the proof of 
 \eqref{;pstrong}. Moreover, if we choose 
 \begin{align*}
  \cf (x)=\frac{1}{(x+\kp )^{e^{-2s}}}, \quad x>0, 
 \end{align*}
 for a given $s>0$, then we may deduce from \eqref{;egenrhc} 
 together with density arguments that 
 for every nonnegative $f\in \lp{1}{\gss{\D }}$, 
 \begin{align*}
  \norm{f}{1}\ge 
  &\exp \left\{ 
  \int _{\R ^{\D }}\log \left( \ou _{s}f\right) d\gss{\D }
  \right\} 
  \ge \norm{\ou _{t}f}{1-e^{-2(s-t)}} 
 %%\end{align*}
 \intertext{for all $0\le t<s$; in particular, taking $t=0$ leads to } 
 %%\begin{align*}
  \norm{f}{1}\ge 
  &\exp \left\{ 
  \int _{\R ^{\D }}\log \left( \ou _{s}f\right) d\gss{\D }
  \right\} 
  \ge \norm{f}{1-e^{-2s}} , 
 \end{align*}
 which is valid for every $s>0$. In the last two inequalities, 
 the upper bound $\norm{f}{1}$ is a consequence of Jensen's  
 inequality. 
\end{rem}

%%\bigskip 
%%
%%\noindent 
%%{\bf Acknowledgements.} 

%%%%%%%%% References %%%%%%%%%

\end{document}